\documentclass[11pt]{amsart}
\usepackage{amsfonts}
\usepackage{amscd}
\usepackage{amsmath}
\usepackage{amssymb}
\usepackage{amsthm}
\usepackage{epsf}
\usepackage{latexsym}
\usepackage{verbatim}
\usepackage{color}
\input epsf.tex

\usepackage{epstopdf}
\usepackage[pdftex]{graphicx}

\newtheorem{theorem}{Theorem}[section]

\newtheorem{lemma}[theorem]{Lemma}
\newtheorem{corollary}[theorem]{Corollary}

\newtheorem{proposition}[theorem]{Proposition}
\theoremstyle{definition}

\begin{document}
\title{Counting Invariant Components of Hyperelliptic Translation Surfaces}
\author[Kathryn A. Lindsey]{Kathryn A. Lindsey}
\address[Kathryn Lindsey]{ Department of Mathematics, Cornell University\\ Ithaca, NY 14853, USA}
\email{ klindsey@math.cornell.edu}
\maketitle 

\begin{abstract}

The flow in a fixed direction on a translation surface $S$ determines a decomposition of $S$ into closed invariant sets, each of which is either periodic or minimal.  We study this decomposition for translation surfaces in the hyperelliptic connected components $\mathcal{H}^{hyp}(2g-2)$ and $\mathcal{H}^{hyp}(g-1,g-1)$ of the corresponding strata of the moduli space of translation surfaces.  Specifically, we characterize the pairs of nonnegative integers $(p,m)$ for which there exists a translation surface in $\mathcal{H}^{hyp}(2g-2)$ or $\mathcal{H}^{hyp}(g-1,g-1)$ with precisely $p$ periodic components and $m$ minimal components.

This extends results by Naveh (\cite{Naveh}), who obtained tight upper bounds on the numbers of minimal components and invariant components a translation surface in any given stratum may have.   Analogous results for the other connected components of moduli space are forthcoming.  

 \end{abstract}

\section{Introduction}
A translation surface is a  compact, connected, oriented, surface endowed with a flat metric which has a finite number of conical singularities, each of which has a cone angle that is an integer multiple of $2\pi$.  Such a surface has trivial linear holonomy, so the flow in a fixed direction on the surface is defined for every regular point.  The flow in a fixed direction on a translation surface $S$ determines a decomposition of $S$ into a finite number of closed, invariant sets called components.  There are two types of components: minimal components and periodic components.  A periodic component is the closure of a maximal cylinder of periodic orbits.  A minimal component is the closure of a recurrent, non-periodic orbit.   Without loss of generality, we will always assume the flow to be in the positive vertical direction.   For all $g \in \mathbb{N}$, the strata $\mathcal{H}(2g-2)$ and $\mathcal{H}(g-1,g-1)$ of the moduli space of translation surfaces each have a connected component called the hyperelliptic component, denoted by $\mathcal{H}^{hyp}(2g-2)$ and $\mathcal{H}^{hyp}(g-1,g-1)$, respectively. 

\begin{theorem} \label{t:maintheorem}
Fix $g \in \mathbb{N}$.  Let $(p,m)$ be a pair of nonnegative integers, at least one of which is nonzero. 
\begin{enumerate}
\item There exists a translation surface in the hyperelliptic component $\mathcal{H}^{hyp}(2g-2)$ with precisely $p$ periodic components and $m$ minimal components if and only if $$3m+2p-1 \leq 2g.$$  
\item There exists a translation surface in the hyperelliptic component $\mathcal{H}^{hyp}(g-1,g-1)$ with precisely $p$ periodic components and $m$ minimal components if and only if $$3m+2p-2 \leq 2g,$$ except for the case $g=1$. 
\end{enumerate}
\end{theorem}

Theorem \ref{t:maintheorem} extends work by Naveh (\cite{Naveh}), who obtained tight upper bounds on $m$ and, for each fixed value of $m$, on $m+p$, with the bounds taken over each \emph{stratum}.  An interesting feature of Theorem \ref{t:maintheorem} is that when trying to maximize the number of invariant components of surfaces in the hyperelliptic components, minimal components ``count" one and a half times as much as periodic components do against an upper bound.  Naveh's bounds in the cases of the strata we are considering are as follows:
\begin{enumerate}
\item Over the stratum $\mathcal{H}(2g-2)$, the tight upper bound on $m$ is $g-1$, and for each fixed $0 \leq m \leq g-1$ the tight upper bound on $m+p$ is $g$. 

 \item Over the stratum $\mathcal{H}(g-1,g-1)$, the tight upper bound on $m$ is $g$.   When $m=g$, $p=0$. For each fixed value of $m$ with $0 \leq m \leq g-1$, the tight upper bound on $m+p$ depends on the parity of $g$.  If $g$ is odd, the tight upper bound is $m+p \leq g+1$.  If $g$ is even and $m=g-1$, the tight upper bound on $p$ is $1$.  If $g$ is even and $m < g-1$, the tight upper bound on $m+p$ is $g+1$.    
\end{enumerate}

The other property which, along with hyperellipticity, is used in \cite{KontsevichZorich} to classify the connected components of strata is the parity of the spin structure of a translation surface.  This quantity is defined for translation surfaces in strata whose surfaces have all even order cone points.  The fact that Naveh's bounds for strata of the form $\mathcal{H}(g-1,g-1)$ depend on the parity of $g$ is related to the existence of the parity of the spin structure.  An analysis of the interplay between spin structures (and the connected components of moduli space they characterize) and numbers of invariant components of translation surfaces will be the subject of a forthcoming work.

\subsection*{Acknowledgments}
I wish to thank my advisor, John Smillie, for his guidance and support in the development of this paper.  

\subsection*{Background}

A \textit{marked translation surface} is a triple $f:(S,\Sigma) \rightarrow (M,\Sigma ')$ consisting of (i) a topological surface $S$ with a finite set of marked points $\Sigma$, (ii) a 2-dimensional real compact manifold $M$ with a set of marked points $\Sigma '$ such that on $M - \Sigma '$ the transition maps between charts are translations, and (iii) a homeomorphism $f: S \rightarrow M$. A \textit{marked half-translation surface} is defined similarly, but the transitions maps between charts in (iii) are only required to be of the form $z \mapsto \pm z + c$.  Equivalently, we can define a marked translation surface to be a compact, connected Riemann surface $R$ together with an Abelian differential $\omega$ defined on $R$; the flat metric is locally defined by integrating $\omega$ along paths.  A marked half-translation surface corresponds to a meromorphic quadratic differential whose poles are simple.  The survey articles \cite{MasurTabachnikov} and \cite{ZorichSurvey} provide an introduction to the theory of flat surfaces.  

The space of marked translation surfaces has a natural stratification based on the number and order of cone points of the surfaces.  Let $k_1,...,k_n$ be natural numbers such that $\sum_{i=1}^n k_i = 2g-2$ for some $g \in \mathbb{N}$, and let $S$ be a topological surface of genus $g$ with $n$ marked points.  Define an equivalence relation $\sim$ as follows: two translation surfaces $f:(S,\Sigma) \rightarrow (M_1,\Sigma_1)$ and $g:(S,\Sigma) \rightarrow (M_2,\Sigma_2)$ are in the same equivalence class if there exists a translation equivalence $h:M_1 \rightarrow M_2$ such that $h \circ f$ is isotopic to $g$.  The stratum $(k_1,...,k_n)$ of the \textit{ space of marked surfaces}, $\tilde{\mathcal{H}}(n_1,...,n_1)$ consists of the set of surfaces translation surfaces $f:(S,\Sigma) \rightarrow (M,\Sigma ')$ whose cone points have orders $n_1,...,n_k$, modulo the equivalence relation $\sim$.   

The \textit{mapping class group }$MCG(S,\Sigma)$ of a marked topological surface $(S,\Sigma)$ is the group of homeomorphisms of $(S,\Sigma)$ (homeomorphisms of $S$ that preserve $\Sigma$ as a set) up to isotopy.  $MCG(S,\Sigma)$ acts on $\tilde{\mathcal{H}}(S,\Sigma)$ via precomposition -- i.e. if $[\alpha] $ is in $MCG(S,\Sigma)$ and $f:(S,\Sigma) \rightarrow (M,\Sigma')$ is a marked translation surface, the image of this surface under $[\alpha]$ is $f \circ \alpha:(S,\Sigma) \rightarrow (M,\Sigma ')$.  The stratum  $\mathcal{H}(k_1,...,k_n)$ of the \textit{moduli space of translation surfaces} consists of $\tilde{\mathcal{H}}(k_1,...,k_2)$ modulo the action of $MCG(S,\Sigma)$.  

Given a path $\gamma$ on a translation surface, define the holonomy coordinates of $\gamma$ to be $(\int_{\gamma} dx, \int_{\gamma} dy)$.  Holonomy coordinates determine a map from $\mathcal{H}$ to $H^1(S,\Sigma; \mathbb{R}^2)$.  We may think of an element of $H^1(S,\Sigma; \mathbb{R}^2)$ as assigning a number in $\mathbb{R}^2$ to each homotopy (rel $\Sigma$) class of paths between points in $\Sigma$.  Given a translation surface $f:(S,\Sigma) \rightarrow (M,\Sigma ')$, define the corresponding element $H^1(S,\Sigma; \mathbb{R}^2)$ to be the element that assigns to each homotopy class of paths in $S$ the holonomy coordinates of the path $f\circ \gamma$ in $M$.  Thus, holonomy defines a map from $\tilde{H}(S,\Sigma) \rightarrow H^1(S,\Sigma; \mathbb{R}^2) \simeq \mathbb{R}^{2(2g+|\Sigma| - 1)}$.  Locally this map is injective -- a slightly perturbed surface has the same combinatorial triangulation, but the sidelengths of the triangles are slightly different.  This map defines a (local) topology and coordinates on $\tilde{\mathcal{H}}(S,\Sigma)$.  The stratum  $\mathcal{H}(S,\Sigma)$ of moduli space can then be equipped with the quotient topology.  The stratum $\mathcal{H}(S,\Sigma)$ is an orbifold of real dimension $2(2g+|\Sigma| - 1)$.

For every $g \in \mathbb{N}$, the strata $\mathcal{H}(2g-2)$ and $\mathcal{H}(g-1,g-1)$ each have a connected component which is called the hyperelliptic component of that stratum and denoted $\mathcal{H}^{hyp}(2g-2)$ and $\mathcal{H}^{hyp}(g-1,g-1)$, respectively.  Let $\mathcal{Q}(k_1,...,k_n)$ denote the moduli space of pairs $(C,\omega)$ where $C$ is a smooth compact complex curve and $\omega$ is a meromorphic quadratic differential on $C$ with zeros of orders $l_1,...,l_n$ such that $\omega$ is not the square of an Abelian differential.  Each pair $(C,\omega) \in \mathcal{Q}(l_1,...,l_n)$ is associated in a canonical way with another connected Riemann surface $C'$ and Abelian differential $\omega '$ on $C$.  Specifically, $C'$ is the unique double covering of $C$ (possibly ramified at the singularities of $\omega$) such that the pullback of $\omega$ is the square of an Abelian differential;  $\omega ' $ is this Abelian differential.  We are interested in the following two special cases of this map:
\begin{align}
\mathcal{Q}(-1^{2g^{\prime}+1},2g^{\prime}-3) &\rightarrow \mathcal{H}(2g^{\prime}-2) \nonumber \\
\mathcal{Q}(-1^{2g^{\prime}+2},2g^{\prime}-2) &\rightarrow \mathcal{H}(g^{\prime}-1,g^{\prime}-1).  \nonumber
\end{align}
(The exponential notation $\mathcal{Q}(-1^{2g^{\prime}+1},2g^{\prime}-3)$ refers to quadratic differentials with $2g^{\prime}+1$ simple poles and one zero of order $2g^{\prime}-3$.)  In these two cases, the map is an injective immersion, the dimension of the domain equals the dimension of the range, and the domain is a nonempty, connected stratum whose elements are topological spheres.  The hyperelliptic component $\mathcal{H}^{hyp}(2g-2)$ consists of the image of the map $\mathcal{Q}(-1^{2g+1},2g-3) \rightarrow \mathcal{H}(2g-2)$.  The hyperelliptic component $\mathcal{H}^{hyp}(g-1,g-1)$ consists of the image of the map $\mathcal{Q}(-1^{2g+2},2g-2) \rightarrow \mathcal{H}(g-1,g-1)$.  Kontsevich and Zorich classified the connected components of the moduli space of translation surfaces in \cite{KontsevichZorich}.  

Each surface in any hyperelliptic component admits a unique hyperelliptic involution - an involution which is an isometry (with respect to the flat metric) and fixes precisely $2g+2$ points, where $g$ is the genus of the surface.  See \cite{Miranda} for an exposition of the theory of hyperelliptic Riemann surfaces, including a proof of the uniqueness of the hyperelliptic involution (page 204).  We will refer to those points in a hyperelliptic surface which are fixed by the hyperelliptic involution and are not cone points as \textit{Weierstrass points}; the Weierstrass points correspond to cone points of angle $\pi$ (or, equivalently, simple poles of the quadratic differential) in the half-translation surface which is the quotient by the hyperelliptic involution. 

A \textit{saddle connection} in a translation surface is a geodesic path in the surface whose end points are (not necessarily distinct) cone points of the surface.  We will call a saddle connection whose midpoint is a Weierstrass point a \textit{Weierstrass edge}.  The boundary of any invariant component of a translation surface is a finite union of saddle connections in the direction of the flow.  

The Poincar\'{e}-Bendixson Theorem implies that every minimal component of a translation surface has genus at least one. A classical result states that a continuous flow on a closed, orientable surface (not necessarily a translation surface) of genus $g$ has at most $g$ distinct sets which are orbit closures of non-periodic, recurrent points.  Moreover, any such surface admits a continuous flow which achieves this bound \cite{Markley}.  Naveh discovered the following two theorems.
 \begin{theorem} (\cite{Naveh}) \label{NavehM}
Let $\mathcal{H} = \mathcal{H}(a_1,...,a_j)$ be a stratum in the moduli space of translation surfaces of genus $g$.
\begin{enumerate} 
\item If $a_i \leq g-1$, $i = 1,2,...j,$ then for every flat surface in $\mathcal{H}$ an upper bound on the number of minimal components is $g$, and this bound is tight. 
\item Otherwise, for every flat surface in $\mathcal{H}$ an upper bound on the number of minimal components is $g-1$, and this bound is tight. \\ \end{enumerate}
\end{theorem}

\begin{theorem} (\cite{Naveh}) \label{NavehMPlusP}
Let $\mathcal{H} = \mathcal{H}(a_1,...,a_j)$ be a stratum in the moduli space of translation surfaces of genus $g \geq 2$.  Denote $B = \{i: a_i \textrm{ is odd}\}$.  Fix $0 \leq M \leq g-1$ and denote $m = \max\{0, M-[g-1 - |B|/2]\}$.  Let $M$ denote the number of minimal components, and $P$  denote the number of periodic components of a translation surface. Then for every flat surface in $\mathcal{H}$ \[M +P \leq g-1 + j -m,\] and this bound is tight (meaning there exists a surface in $\mathcal{H}$ with $M$ minimal components and $g-1+j-m - M$ periodic components).  
If $M = g$, then $P = 0$ and $M+P = g$. 
\end{theorem}

\section{Invariant component diagrams and upper bounds}

Throughout, we will assume the flow on a surface to be in the vertical direction.

\begin{lemma} \label{l:hypInvolFixesInvariantComponents} Let $S$ be an element of  $\mathcal{H}^{hyp}(2g-2)$ or $\mathcal{H}^{hyp}(g-1,g-1)$ for any $g \in \mathbb{N}$, and denote by $\gamma$ the hyperelliptic involution of $S$. Then $\gamma(C)=C$ for every invariant component $C$ of $S$.  
\end{lemma}

\begin{proof}  If $g=1$, the translation surface $S$ has no saddle connections and so consists of a single invariant component.  If $S$ consist of a single invariant component the conclusion is immediate, regardless of $g$.   

So assume $g \geq 2$ and assume $S$ has more than one invariant component (in the vertical direction).   Let $A$ be the maximum of the areas (with respect to the flat metric) of the invariant components of $S$.  For $t > 0$, let $M_t \in GL_2(\mathbb{R})$ be the matrix $\bigl(\begin{smallmatrix}
t&0\\ 0&1
\end{smallmatrix} \bigr)$.  For $t>0$, denote by $S_t$ the surface made by gluing together $S-C$ and $M_t(C)$ in the same way as in the original surface $S$. (The boundary of $C$ in $S$ consists of vertical saddle connections, and the matrices $M_t$ preserve the direction and lengths of vertical saddle connections, so we can still glue along the vertical saddle connections to form $S_t$.)   

Pick $p>1$ such that the area of $M_p(C)$ is greater than $A$. The surfaces $S_p$ and $S$ are in the same connected component of moduli space, the hyperelliptic component, since $\{S_t : 1 \leq t \leq p \}$ forms a path in moduli space from $S$ to $S_p$.  Let $\gamma_p$ be the hyperelliptic involution of $S_p$.  Since $\gamma_p$ is an isometry which maps vertical geodesics to vertical geodesics, $\gamma_p$ maps minimal components to minimal components and periodic components to periodic components (of the same area). Since $area(M_p(C)) > A$, we must have $\gamma_p(M_p(C))=M_p(C)$.  Hence $\gamma_p(S_p - M_p(C)) = S_p-M_p(C)$.  

Now we will use $\gamma_p $ to construct a hyperelliptic involution $\gamma_1$ on $S$.  On $S-C$, let $\gamma_1 = \gamma _p$.  On $C$, let $\gamma_1 =  M_p^{-1} \circ \gamma_p \circ  M_p$.  Since $\gamma_p$ is an involution and has $2g-2$ fixed points, so does $\gamma_1$.  Thus $\gamma_1$ is a hyperelliptic involution on $S$.  Since hyperelliptic involutions are unique for surfaces of genus $ g \geq 2$, $\gamma = \gamma_1$.  Thus $\gamma(C) = \gamma_1(C) = C$.  
\end{proof}

For any invariant component $C$, we will refer to a vertical saddle connection (recall we are assuming the flow to be in the vertical direction) in the boundary of $C$ as a \textit{boundary edge} of $C$.  

\begin{lemma} \label{l:PairedEdges}
Let $C$ and $D$ be invariant components of a surface $S$ in $\mathcal{H}^{hyp}(2g-2)$ or $\mathcal{H}^{hyp}(g-1,g-1)$. Then the number of boundary edges of $C$ that are glued to $D$ is even, and the hyperelliptic involution interchanges pairs of these boundary edges.
\end{lemma}

\begin{proof}
Assume there exists a boundary edge $e_1$ of $C$ which is glued to $D$.  Let $p$ be the midpoint of $e_1$ and let $\gamma$ be the hyperelliptic involution of $S$.  Suppose $\gamma(e_1) = e_1$.  Then since $\gamma$ is an isometry, $\gamma$ restricted to $e_1$ is either the identity or rotation about the midpoint of $e_1$.  Since $\gamma$ has only finitely many fixed points in $S$, $\gamma$ cannot be the identity.  Hence $p$ is a Weierstrass point,  and in a small neighborhood of $p$ in $S$,  $\gamma$ acts as a rotation around $p$ by $\pi$ radians.   Since $\gamma(C)=C$ by Lemma \ref{l:hypInvolFixesInvariantComponents}, this means a small disk about $p$ is contained in $C$, contradicting the fact that $p \in e_1$ is in the boundary of $C$.  

Therefore $\gamma(e_1) \not = e_1$.  Let $e_2 = \gamma(e_1)$.  Since both $C$ and $D$ are fixed by $\gamma$, $e_2$ is also a boundary edge of $C$ that is glued to $D$.  Since $\gamma$ is an involution $\gamma(e_2) = e_1$. 
\end{proof}

We will now define the dissection $\hat{S}$ of a hyperelliptic surface $S$ in $\mathcal{H}^{hyp}(2g-2)$ or $\mathcal{H}^{hyp}(g-1,g-1)$.  The first step in constructing $\hat{S}$ is to cut along all boundary edges of of all the invariant components of $S$.  This results in a number of ``pieces" (surfaces with boundary) --  one for each invariant component of $S$.  Each boundary edge of an invariant piece $C$ is paired with another boundary edge of $C$ by the hyperelliptic involution by Lemma \ref{l:PairedEdges}.  The second  step in constructing $\hat{S}$ is to glue the paired edges together via translations.  This yields a union of closed surfaces, say $S_1,...,S_k$, where $k$ is the number of invariant components of $S$.  We will call each surface $S_i$ a \emph{piece} of the dissection $\hat{S}$. (Lemma \ref{l:piecesAreTransSurfs} will show that each $S_i$ is a translation surface, as opposed to a half-translation surface). Define the set of \emph{augmented cone points} to be the set consisting of all preimages in $S_1 \cup ... \cup S_k$ of the cone points of $S$.  (Every cone point in any surface $S_i$ is the preimage of a cone point of $S$, but not every preimage in a $S_i$ of a cone points of $S$ is a cone point (i.e. has cone angle $>2\pi$) of $S_i$.)  The \textbf{dissection} $\hat{S}$ of $S$ is of the collection of pieces $\{S_1,...,S_k\}$ together with the set of augmented cone points.  

For example, suppose $S$ is a surface in $\mathcal{H}(1,1)$ which consists of two minimal tori glued along a vertical slit.   To construct the dissection $\hat{S}$, the first step would cut along the slit.  This yields two tori, each of which has a slit cut in it.  Second, since the hyperelliptic involution interchanges the sides of each slit, we would ``heal" (glue together the sides of) the slit in each torus.  This yields two minimal tori $S_1$ and $S_2$ (without slits); these are the the two pieces of $\hat{S}$.  The set of augmented cone points consists of four points: the points in each torus that were the top and bottom of the slit.  The dissection $\hat{S}$ consists of the two pieces $S_1$ and $S_2$, along with the set of the four augmented cone points.  

Define an \textit{augmented saddle connection} of $\hat{S}$ to be a geodesic path in one of the $S_i$ whose endpoints are both augmented cone points.  If $C$ is an invariant component of $S$, each pair of boundary edges of $C$ becomes a vertical augmented saddle connection in the piece $S_C$ in $\hat{S}$ corresponding to $C$.   Furthermore, the restriction of the hyperelliptic involution to $S_C$ defines an isometric involution on $S_C$ for which the midpoint of this augmented saddle connection is a fixed point.

\begin{lemma} \label{l:piecesAreTransSurfs}
Each piece $S_i$ of the dissection $\hat{S}$ is a translation surface.
\end{lemma}

\begin{proof}
To show that $S_i$ is a translation surface, as opposed to a half-translation surface, it suffices to show that when we describe $S_i$ as a finite collection of polygons embedded in $\mathbb{R}^2$ whose boundaries are given a counter-clockwise orientation, edges which are identified are parallel and \textit{have opposite orientations}.  (Edge identifications for polygons comprising a half-translation surface are not required to identify oppositely-oriented edges.)  

Represent $S$ by a finite collection $P$ of polygons embedded in $\mathbb{R}^2$ whose boundaries are oriented counter-clockwise, along with ``gluing rules" which identify pairs of parallel, oppositely-oriented edges.  Without loss of generality, we may assume the hyperelliptic involution interchanges pairs of congruent polygons.  We may further assume that each polygon is contained in a unique invariant component of $S$.  Let $P_i \subset P$ be the subset consisting of polygons which make up the invariant component of $S$ which corresponds to piece $S_i$.  

Every edge identification of polygons in $P$ satisfies the requirement that edges have opposite orientation.  The only edges of polygons in $P_i$ which are not glued (in $S$) to other edges of polygons in $P_i$ are those that belong to the boundary of the invariant component corresponding to the piece $S_i$.  To form the piece $S_i$, we glue (via translations) pairs of these edges interchanged by the hyperelliptic involution.  Thus, the restriction of the hyperelliptic involution to $S_i$ fixes the midpoint of such a pair of identified edges, and in a neighborhood of this point, acts as a rotation by $\pi$ radians about this point.  Since the hyperelliptic involution interchanges pairs of congruent polygons, this implies that the embeddings in $\mathbb{R}^2$ of paired congruent polygons differ by a rotation by $\pi$ radians.  Consequently, if $e_1$ and $e_2$ are edges of polygons in $P_i$ such that $e_1$ and $e_2$ are contained the boundary edges of the invariant component $S_i \subset S$ and $e_1$ and $e_2$ are paired  by the hyperelliptic involution, then $e_1$ and $e_2$ have opposite orientations.  Therefore $S_i$ is a translation surface.  
\end{proof}

\begin{lemma} \label{l:propertiesOfPieces}
Let $S$ be an element of $\mathcal{H}^{hyp}(2g-2)$ or $\mathcal{H}^{hyp}(g-1,g-1)$.  
 \begin{enumerate} 

\item \label{it:quotientPieceGenus0} The quotient of each piece $S_i$ of the dissection $\hat{S}$ by the isometric involution that is the restriction of of the hyperelliptic involution of $S$ to the piece $S_i$ has genus 0. 
\item \label{it:tree} The pieces of $\hat{S}$ are arranged in a tree. 
 \end{enumerate}
\end{lemma}

\begin{proof}
Let $S_1$,...,$S_n$ be the pieces of $\hat{S}$.  Let $S^*$ denote the quotient surface $S/\gamma$, and let $S_1^*,...,S_n^*$ denote the quotient of the pieces by the isometric involutions which are the restrictions of the hyperelliptic involution $\gamma$ of $S$ to each piece.  The surface $S$ is formed from the pieces $S_1,...,S_n$ by cutting along and gluing pairs of augmented saddle connections in the pieces.  Let $p_1,...,p_m$ be a list of the pairs of augmented saddle connections which are joined together to form $S$. Let $p_1^*,...,p_m^*$ be a list of the pairs of geodesic segments in $S_1^*,...,S_n^*$ which are the images of $p_1,...,p_m$.  Then the quotient surface $S^*$ is formed from $S_1^*,...,S_n^*$ but cutting along and gluing the pairs $p_1^*,...,p_m^*$.  
 
 We will express the Euler characteristic $\chi(S^*)$ in terms of the Euler characteristics $\chi(S_1^*),...,\chi(S_n^*)$.  We have that $\chi(S^*)=2$ by the definition of a hyperelliptic surface.  The surface obtained by cutting along and gluing a pair of segments  $p_i^*$ is homeomorphic to the connected sum of the two quotient pieces.  The Euler characteristic of the connected sum of any two surfaces $X_1$ and $X_2$ equals $\chi(X_1)+\chi(X_2)-2$.  Therefore,

\begin{equation} \label{eq:part1} 
2= \chi(S^*) = -2m + \sum_{i=1}^n \chi(S_i^*). \end{equation}
Denote the genus of $S_i^*$ by $g_i^*$. Then $\chi(S_i^*) = 2-2g_i^*$.  In order to connect all the pieces $S_1,...,S_n$, we must have that $m \geq n-1$.  Thus,
 \begin{equation} \label{eq:part2} -2m + \sum_{i=1}^n \chi(S_i^*) \leq -2(n-1) + \sum_{i=1}^n \chi(S_i^*) = 2 - \sum_{i=1}^n 2g_i^*. \end{equation}
Combining equations (\ref{eq:part1}) and (\ref{eq:part2}) yields $$2 \leq 2-\sum_{i=1}^n 2g_i^* ,$$ implying $g_i^* = 0$ for all $i$.  This proves part \ref{it:quotientPieceGenus0} of the proposition.  

Using the fact that $g_i^* = 0$ for all $i$, we have 
$$2 = \chi(S^*) = -2m + \sum_{i=1}^n \chi(S_i^*) = -2m + 2n - \sum_{i=1}^n 2g_i^* = 2(n-m).$$  Hence, $m= n-1$.  Since we have $n$ pieces connected along $m=n-1$ slits, the pieces must be arranged as a tree, proving part \ref{it:tree} of the proposition. 
\end{proof}

Combining Lemmas \ref{l:piecesAreTransSurfs} and \ref{l:propertiesOfPieces} yields the following classification of surfaces in the hyperelliptic components $\mathcal{H}^{hyp}(2g-2)$ and $\mathcal{H}^{hyp}(g-1,g-1)$:

\begin{corollary}
Each translation surface in $\mathcal{H}^{hyp}(2g-2)$ or $\mathcal{H}^{hyp}(g-1,g-1)$ with $n$ invariant components has a canonical description as consisting of $n$ hyperelliptic translation surfaces which have slits cut along augmented Weierstrass edges, and these hyperelliptic translation surfaces are glued together in a tree configuration by identifying pairs of slits via translations.  
\end{corollary}

We now define the \textbf{invariant component diagram} associated to a surface $S$ in $\mathcal{H}^{hyp}(2g-2)$ or $\mathcal{H}^{hyp}(g-1,g-1)$, a graph-like object that describes how the pieces of $\hat{S}$ can be connected to form $S$.  The graph has $n$ vertices, where $n$ is the number of pieces of $\hat{S}$, and a bijection relates the vertices with the pieces of $\hat{S}$.  The graph has a number of \textit{half-edges}, line segments which have one end at a vertex and the other end free-floating.  There are two different types of half-edges, which we will call solid half-edges and dotted half-edges.  Two solid half-edges incident to two distinct vertices may joint to form a full solid edge.  

Each solid half-edge incident to a vertex represents a vertical augmented saddle connection on the corresponding piece of $\hat{S}$.  Each dotted half-edge incident to a vertex represents a pair of augmented critical leaves of the vertical foliation of the corresponding piece of $\hat{S}$ which do not form an augmented saddle connection (i.e. a dotted half-edge represents a ``broken" augmented saddle connection).  Define a bijection relating the half-edges of the graph and the augmented saddle connections (intact or ``broken") of the pieces of $\hat{S}$.  

To form $S$ from the pieces of $\hat{S}$, we cut along some of the vertical augmented saddle connections in the pieces, turning these augmented saddle connections into slits, and then we glue together pairs of slits.  For each pair of these slits which are glued together to form $S$, connect the corresponding two solid half-edges in the graph so that they form a full solid edge.  

\medskip

{\bf Algorithmic definition of the minimal component diagram associated to $S$:}
\begin{enumerate}
\item Draw $n$ labeled vertices, where $n$ is the number of pieces of the dissection $\hat{S}$. Say that vertex $v_i$ represents piece $C_i$ of the dissection $\hat{S}$.
\item For each $i$, draw $k_i$ solid half-edges incident to vertex $v_i$, where $k_i$ is the number of augmented vertical saddle connections in the surface $C_i$.  
\item For each $i$, draw $j_i$ dotted half-edges incident to vertex $v_i$, where $j_i$ is the number of augmented ``broken" vertical saddle connections in the surface $C_i$. (A ``broken" saddle connection means a pair of critical leaves of the vertical foliation which are of infinite length and do not form a saddle connection.)
\item For each pair of vertical saddle connections in the original surface $S$ which were boundary edges of components $C_m$ and $C_n$ (and which we cut when forming the dissection $\hat{S}$), connect a pair of solid half-edges incident to vertices $v_m$ and $v_n$ to form a full solid edge between these two vertices.  
\end{enumerate}

  \begin{figure}[hb] 
  \centering
  \includegraphics[width= \linewidth]{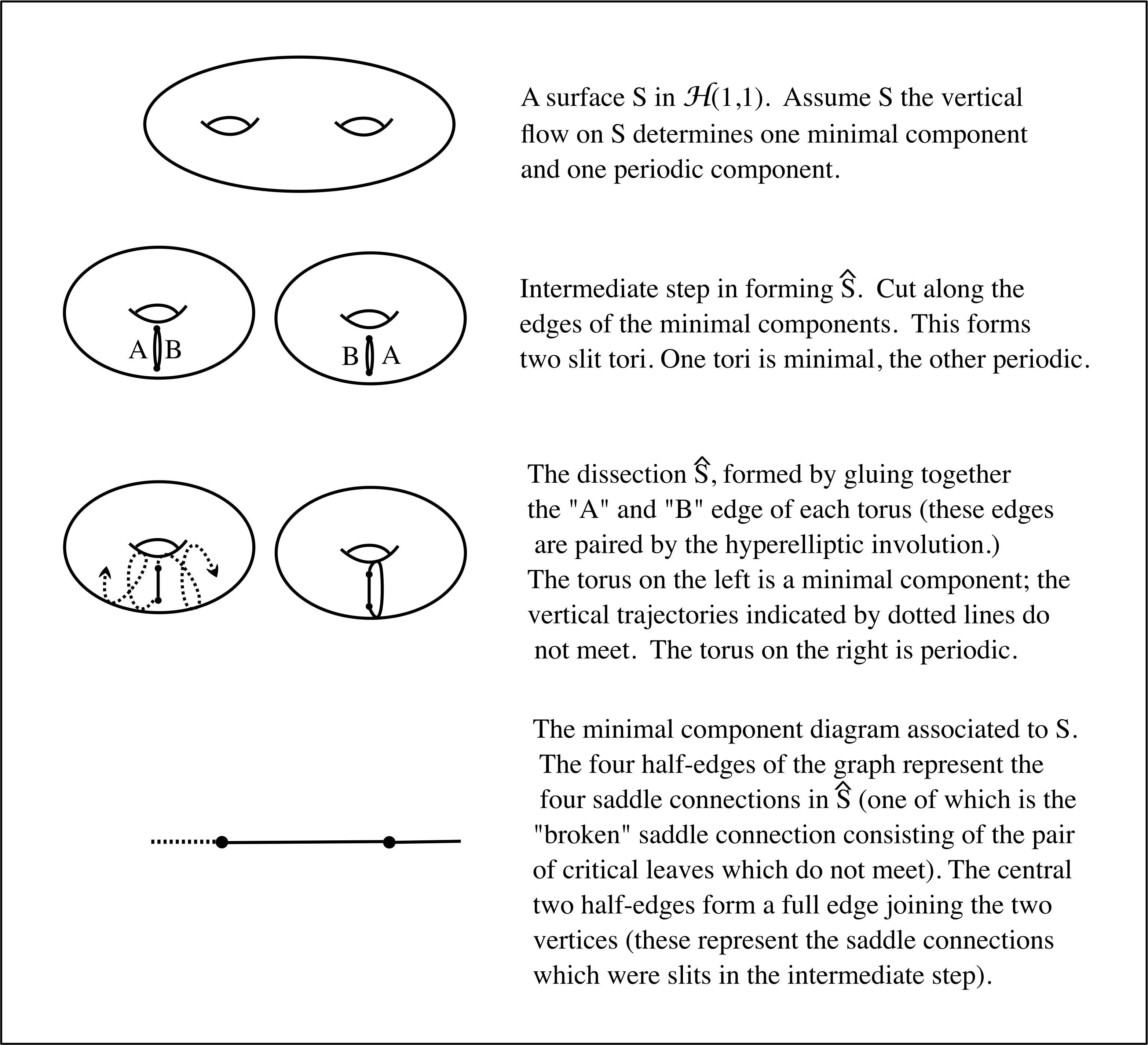}
  \caption[]
  {Sample invariant component diagram.}  
  \label{f:algorithm}
  \end{figure}

\begin{lemma} \label{l:totalConeAngle}
Let $S$ be in $\mathcal{H}^{hyp}(2g-2)$ or $\mathcal{H}^{hyp}(g-1,g-1)$.  Denote by $n_i$ the number of half-edges incident to the $i^{\textrm{th}}$ vertex of the invariant component diagram associated to $S$.  Then the total cone angle of $S$ equals $ 2\pi \sum_i n_i$.
\end{lemma}

\begin{proof}
A cone point with cone angle $k\pi$ has $k$ vertical rays (going either ``up" or ``down") which emanate from the cone point.  A vertical (possibly augmented) saddle connection requires two such rays (one going up, one going down, meeting in the middle).  A ``broken" saddle connection (i.e. a pair of non-closed augmented critical leaves of the vertical foliation) also requires two such rays.  Each half-edge incident to a vertex in the invariant component diagram represents a vertical augmented (possibly broken) saddle connection in the corresponding piece of the dissection $\hat{S}$.  Hence, $2\pi$ multiplied by the number of half-edges incident to a fixed vertex equals the sum over the augmented cone points in that piece of the cone angle at those points.  Hence, the total cone angle of $S$ equals the sum over all the augmented cone points in all the pieces in $\hat{S}$ of the cone angle at those points.  
\end{proof}

\begin{lemma} \label{l:allPossibleSaddleConnections}
Let $S$ be a translation surface such that every critical leaf of the vertical foliation is closed.  Then $S$ admits a cylinder decomposition in the vertical direction. 
\end{lemma}

\begin{proof}
Assume every critical leaf of the vertical foliation of $S$ is closed.  Let $x$ be a point in $S$ that is located some small distance $\epsilon$ in the horizontal direction from a vertical saddle connection.  Then every point of the leaf passing through $x$ of the vertical foliation is also $\epsilon$ away from a vertical saddle connection.  Because $S$ has finitely many vertical saddle connections, each of which is of finite length, the leaf passing through $x$ must be periodic.  As $x$ was arbitrary, this implies every vertical saddle connection has a neighborhood of periodic points.  Since the boundaries between invariant components of $S$ are vertical saddle connections, this implies every regular point of $S$ is periodic.  
\end{proof}

Dotted half-edges in the invariant component diagram represent ``broken" saddle connections, so Lemma \ref{l:allPossibleSaddleConnections} immediately implies:

\begin{corollary}
Each vertex in the invariant component diagram which corresponds to a minimal component in $S$ has at least one incident dotted half-edge.  
\end{corollary}

\begin{lemma} \label{l:ConeAngle}
Let $S$ be a translation surface in $\mathcal{H}^{hyp}(2g-2)$ or $\mathcal{H}^{hyp}(g-1,g-1)$ which has $m$ minimal components and $p$ periodic components.  Then the total cone angle of $S$ is at least $2\pi(3m+2p-2)$.  
\end{lemma}

\begin{proof}
The invariant component diagram associated to $S$ must have
\begin{enumerate}
\item $m+p$ vertices (of which $m$ represent minimal components and $p$ represent periodic components),
\item at least one dotted half-edge incident to each of the ``minimal" vertices, and
\item enough solid full edges to make the graph connected.
\end{enumerate}
The diagram may have additional half-edges, which may be either solid or dotted.  At a minimum, then, the diagram has $m$ dotted half-edges, and $m+p-1$ solid full edges.  Each of the solid full edges consists of two half-edges.  Thus the total number of half-edges in the invariant component diagram is at least $m+2(m+p-1) = 3m+2p-2$.  By Lemma \ref{l:totalConeAngle}, the total cone angle of $S$ is therefore at least $2\pi(3m+2p-2)$.   
\end{proof}

Proposition \ref{p:upperBounds} follows immediately from Lemma \ref{l:ConeAngle} and the fact that the total cone angle of a surface in $\mathcal{H}(2g-2)$ is $2\pi(2g-1)$ while the total cone angle of a surface in $\mathcal{H}(g-1,g-1)$ is $2\pi(2g)$.  

\begin{proposition} \label{p:upperBounds}
Fix $g \in \mathbb{N}$.  Let $(p,m)$ be a pair of nonnegative integers, at least one of which is nonzero. 
\begin{enumerate}
\item If there exists a translation surface in the hyperelliptic component $\mathcal{H}^{hyp}(2g-2)$ with precisely $p$ periodic components and $m$ minimal components then $$3m+2p-1 \leq 2g.$$  
\item If there exists a translation surface in the hyperelliptic component $\mathcal{H}^{hyp}(g-1,g-1)$ with precisely $p$ periodic components and $m$ minimal components then $$3m+2p-2 \leq 2g.$$ 
\end{enumerate}
\end{proposition}


\section{Constructing surfaces with specific numbers of periodic and minimal components}

We will use invariant component diagrams as ``blueprints" for constructing surfaces in $\mathcal{H}^{hyp}(2g-2)$ and $\mathcal{H}^{hyp}(g-1,g-1)$ with specific numbers of periodic and minimal components.   First, we will construct the ``building blocks" we will be using -- hyperelliptic surfaces with given numbers of Weierstrass edges and ``broken" saddle connections.  

\begin{lemma} \label{l:technical}
Let $X = \mathbb{R}/\mathbb{Z}$ and $Y = \mathbb{R}/\mathbb{Z}$ with the standard ordering on $[0,1)$.  Let $x_0 < x_1 < ... x_{n-1}$ be $n$ distinct points in $X$.  For each $i$, define $y_i \in Y$ by $y_i = -x_i$.  For each $i$, define $\varphi_i:[x_i,x_{i+1 \bmod{n}}]\rightarrow [y_{i+1\bmod{n}},y_i]$ by $$\varphi_i(x)=x+y_{i+1\bmod{n}} - x_i.$$ 
Let $\sim$ be the equivalence relation on $X \cup Y$ generated by the $\varphi_i$'s. 
Then
\begin{enumerate}
\item  if $n$ is odd, $\{x_0,...,x_{n-1},y_0,...,y_{n-1}\}/\sim$ consists of a single equivalence class, 
\item if $n$ is even, $\{x_0,...,x_{n-1},y_0,...,y_{n-1}\}/\sim$ consists of two equivalence classes, each of which has $n$ preimages.  
\end{enumerate}  
\end{lemma}

For each natural number $n$, we will define hyperelliptic surfaces $P_n$ and $M_n$.  $P_n$ will consist of a single periodic component and $M_n$ will consist of a single minimal component.  The surface $P_n$ will have $n$ vertical Weierstrass edges, and $M_n$ will have $n-1$ vertical Weierstrass edges.  (In the cases $n=1$ and $n=2$, $P_n$ and $M_n$ have no true cone points, but we will think of them as having $n$ augmented cone points.  $P_1$ has one vertical augmented Weierstrass edge, $M_1$ has no vertical augmented Weierstrass edges (and will not be used), $P_2$ has two vertical augmented Weierstrass edges, and $M_2$ has one vertical augmented Weierstrass edge.)  If $n$ is even, $P_n$ and $M_n$ will be genus $\frac{n}{2}$ surfaces in $\mathcal{H}^{hyp}(\frac{n}{2}-1,\frac{n}{2}-1)$.  If $n$ is odd, $P_n$ and $M_n$ will be genus $\frac{n+1}{2}$ surfaces in $\mathcal{H}^{hyp}(n-1)$.  

The surfaces $P_n$ and $M_n$ do not have boundary.  To connect two such surfaces together, we will cut along a vertical Weierstrass edge in each, and glue along the resulting slits.    
\medskip

{\bf The surfaces $P_n$:}
For $n \in \mathbb{N}$, define $P_n$ as follows.  Begin with a flat rectangle which measures $n$ units in the vertical direction and $1$ unit in the horizontal direction.  Partition each of the vertical sides into $n$ disjoint segments of length $1$.  On the left side of the rectangle, label the segments $s_1,...,s_n$, in order, with $s_1$ at the top and $s_n$ at the bottom.  On the right side of the rectangle, again label the segments $s_1,...,s_n$, but use the opposite order: label the top segment $s_n$, and the bottom segment $s_1$.   For each $i$, identify the two vertical segments labeled $s_i$ via a translation.  Identify the two horizontal sides of the rectangle via a translation in the vertical direction (so that the vertical sides of the rectangle each become a closed curve).  This surface is $P_n$.  

The only cone point(s) in $P_n$ are formed by identifying the endpoints of the segments $s_i$.  By Lemma \ref{l:technical}, if $n$ is odd all the endpoints of the $s_i$'s are identified to form a single cone point.  If $n$ is even, Lemma \ref{l:technical} shows that the endpoints of the $s_i$'s form two distinct cone points with equal cone angle.  The total cone angle of $P_n$ is $2\pi n$, so $P_n$ has genus $\frac{n+1}{2}$ if $n$ is odd and genus $\frac{n}{2}$ if $n$ is even.  It is easy to see that $P_n$ admits an isometric involution with $2g+2$ fixed points:  rotating the entire rectangle by a half-turn fixes the midpoint of each $s_i$, the midpoint of the rectangle, the midpoint of the horizontal sides of the rectangle, and the unique cone point if $n$ is odd.  Thus, $P_n$ admits a hyperelliptic involution.  

\medskip

{\bf The surfaces $M_n$:} For $n \in M$, define $M_n$ as follows.  Begin with the rectangular representation of the surface $P_n$ defined above.  Perform a vertical shear (apply a matrix $\bigl(\begin{smallmatrix} 1&0\\ t&1 \end{smallmatrix} \bigr)$ ) so that the two segments labeled $s_1$ on the vertical sides of the rectangle are at the same vertical height.  Then the horizontal path from the top of the left $s_1$ segment to the top of the right $s_1$ segment is a closed horizontal curve, as is the horizontal path from the bottom of the left $s_1$ segment to the bottom of the right $s_1$ segment.  Apply an irrational horizontal shear (a matrix $\bigl(\begin{smallmatrix} 1&\alpha\\ 0&1 \end{smallmatrix} \bigr)$ where $\alpha$ is irrational) to the rectangle whose vertical sides are the two $s_1$ segments and whose horizontal sides are the closed horizontal curves connecting the endpoints of the $s_1$ segments.  The resulting surface is $M_n$.  

Since $M_n$ is obtained by continuously deforming $P_n$ and $P_n$ is in the hyperelliptic connected component, $M_n$ is also in the hyperelliptic connected component.  $M_n$ has the same number and type of cone points as $P_n$.  The irrational twist on the rectangle whose vertical edges are $s_1$ destroys one vertical Weierstrass edge and makes the vertical foliation of $M_n$ minimal. 

\medskip

 \begin{figure}[!htpb]
  \centering
  \includegraphics[width= .35\linewidth]{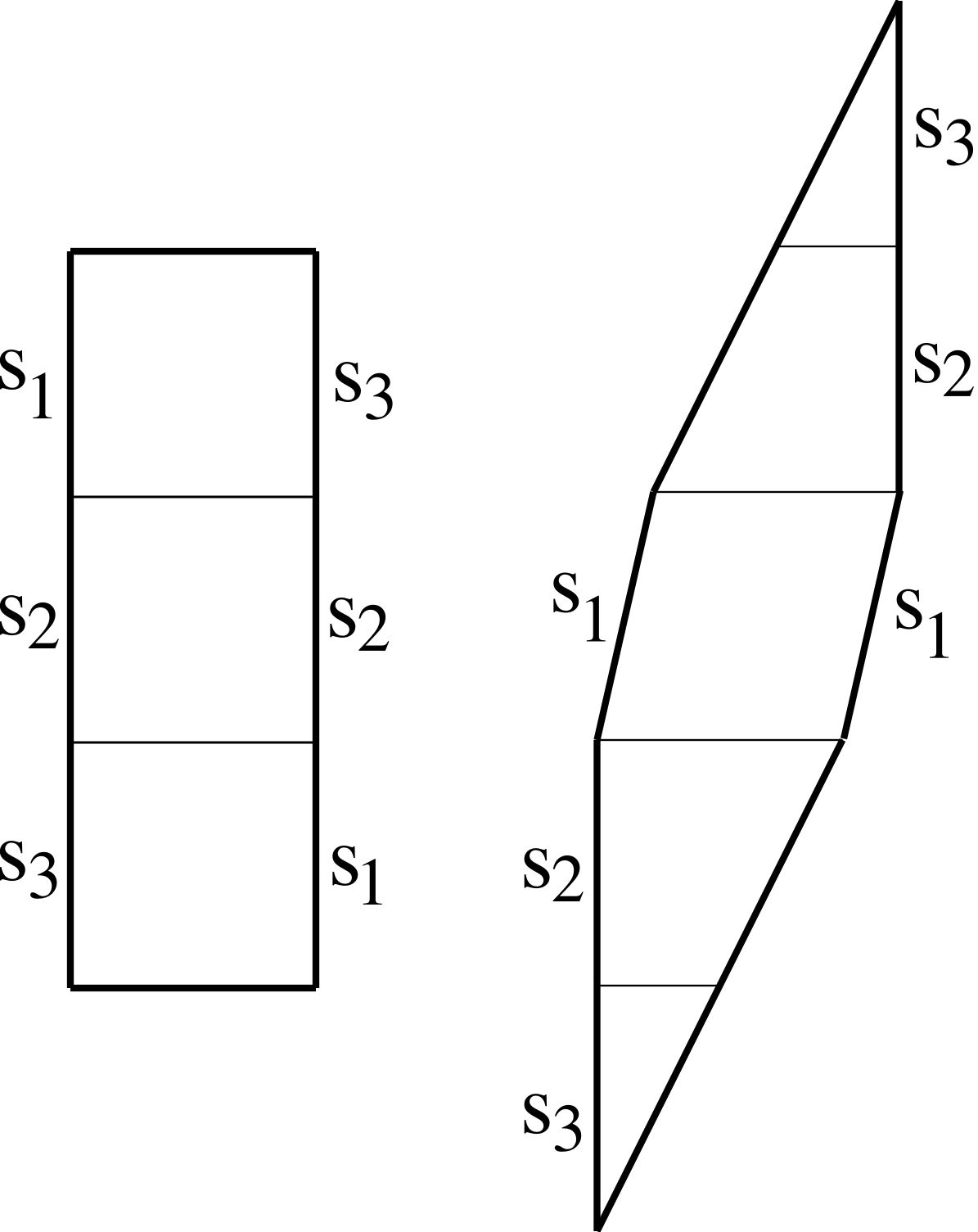}
  \caption[]
  {Polygonal representations of the surfaces $P_3$ (left) and $M_3$ (right).}  
  \label{f:P3M3}
  \end{figure}
  
  \medskip

We now describe how to construct surfaces associated to certain invariant component diagrams (those in which every vertex has at most one incident dotted half-edge).  A vertex with $n \in \mathbb{N}$ incident (half-)edges corresponds to a surface $P_n$ if all (half-)edges are solid and corresponds to a surface $M_n$ if precisely one of the (half-)edges is dotted.  For each solid full edge which connects two vertices, cut along a vertical Weierstrass edge in each to form slits, and glue the slits together via a translation.  (Recall that for $n=1$ and $n=2$, we interpret the surfaces $P_n$ and $M_n$ to have $n$ augmented cone points, and use the augmented Weierstrass edges instead of true saddle connections.)  

We will now define two families of invariant component diagrams.   For each $k \in \mathbb{N}$ and each pair of nonnegative integers $(p,m)$, at least one of which is nonzero and such that $3m+2p-2 \leq k$, we will define two invariant component diagrams: the \emph{p-central diagram} $\mathcal{D}^{per}_{(k,p,m)}$ and the \emph{m-central diagram} $\mathcal{D}^{min}_{(k,p,m)}$.  

Construct the $p$-central diagram $\mathcal{D}^{per}_{(k,p,m)}$ as follows: Draw one central vertex.  Now draw $m+p-1$ other vertices and connect each of these vertices to the central vertex with a full solid edge.  Add a dotted half-edge to $m$ of the non-central vertices.  Add $y=k-(3m+2p-2)$ solid half-edges to the central vertex.  

Construct the $m$-central diagram $\mathcal{D}^{min}_{(k,p,m)}$ as follows:  Draw one central vertex.  Now draw $m+p-1$ other vertices and connect each of these vertices to the central vertex with a full solid edge.  Add a dotted half-edge to the central vertex and to $m-1$ of the non-central vertices.  Add $y=k-(3m+2p-2)$ solid half-edges to the central vertex.  

\begin{lemma} \label{l:numberOfConePts}
Fix an integer $k \geq 3$.  The translation surface associated to an invariant component diagram $\mathcal{D}^{per}_{(k,p,m)}$ or $\mathcal{D}^{min}_{(k,p,m)}$ has one cone point if $k$ is odd and two equal cone points if $k$ is even. 
\end{lemma}

\begin{proof}
Fix any integer $k \geq 3$ and any such diagram $\mathcal{D}$.  Let $p^{\prime}$ be the number of non-central periodic vertices of $\mathcal{D}$, let $m^{\prime}$ be the number of non-central minimal vertices of $\mathcal{D}$, and let $y^{\prime}$ be the number of half-edges (either dotted of solid) incident to the central vertex which are not part of a full edge.  Regardless of whether $\mathcal{D}$ is $m$-central or $p$-central, we have $k=2p^{\prime}+3m^{\prime} + y^{\prime}$.  The central component is either a $P_n$ or a $M_n$ building block for some $n \in \mathbb{N}$.  Each of the $p^{\prime}$ periodic non-central vertices represents a periodic cylinder glued to a Weierstrass edge in the central component.  Gluing a cylinder to a Weierstrass edge in the central component forces the marked points at the top and bottom of the Weierstrass edge to coalesce.  Thus, for each of the $p^{\prime}$ cylinders attached to the central component, the number of marked points on the two vertical boundaries of the rectangle representing the central component decreases by one.  Gluing a minimal slit tori along a Weierstrass edge does not cause any marked points to coalesce.  

The total number of half-edges (counting a full edge as two half-edges) in $\mathcal{D}$ is $k$.  Of these $k$ half-edges, $2p^{\prime}$ are used to connect the central vertex to the non-central periodic vertices, and another $2m^{\prime}$ of the half-edges are incident to non-central minimal vertices (each of the $m^{\prime}$ non-central minimal vertices has a dotted half-edge and a half-edge contained in the full edge connecting the vertex to the central vertex.)  Thus, the number of half-edges incident to the central vertex which do not represent Weierstrass edges whose top and bottom points are forced by cylinders to coalesce is $y^{\prime}+m^{\prime} = k - 2p^{\prime} - 2m^{\prime}$.  This is the number of distinct marked points on each of the two vertical boundaries of the rectangle representing the central component.  By Lemma \ref{l:technical}, these marked points are identified to form one marked point if  $k - 2p^{\prime} - 2m^{\prime}$ is odd and two marked points with equal numbers of preimages if $k - 2p^{\prime} - 2m^{\prime}$ is even.  (The condition that $k \geq 3$ excludes the diagram consisting of two vertices connected by a solid full edge; in this case the total surface is a torus and has no cone points.)  Since $2p^{\prime}+2m^{\prime}$ is even, the total surface therefore has one cone point if $k$ is odd and two equal cone points if $k$ is even.  
  \end{proof}
  
  \begin{corollary} \label{c:realizable}
  Fix an integer $k \geq 3$ and a pair of nonnegative integers $(p,m)$, at least one of which is nonzero and such that $3m+2p-1 \leq k$.  A translation surface surface associated to the invariant component diagram $\mathcal{D}^{per}_{(k,p,m)}$ or $\mathcal{D}^{min}_{(k,p,m)}$ is in $\mathcal{H}^{hyp}(\frac{k}{2}-1,\frac{k}{2}-1)$ if $k$ is even and is in $\mathcal{H}^{hyp}(k-1)$ if $k$ is odd.  Furthermore, such a surface has precisely $p$ periodic components and $m$ minimal components. 
  \end{corollary}
  
  \begin{proof}
Let $S$ be the translation surface associated to one of these diagrams.  Every building block $B$ used to construct $S$ is a hyperelliptic surface with hyperelliptic involution $\gamma_B$ which fixes $2g_B+2$ points, where $g_B$ is the genus of $B$.  When we glue two building blocks $B_1$ and $B_2$ together, we cut a slit along a Weierstrass edge in each; this destroys one fixed point in each building block.  Since $\gamma_{B_1}$ and $\gamma_{B_2}$ agree along the slits, we can define an involution $\gamma_{B_1 \sqcup B_2} $ on the total surface by defining $\gamma_{B_1 \sqcup B_2}$ piecewise as whichever of $\gamma_{B_1}$ and $\gamma_{B_2}$ is defined.  Thus $\gamma_{B_1 \sqcup B_2}$  fixes $2(g_{B_1}+g_{B_2})+2$ points and is the hyperelliptic involution on the total surface.  

Similarly, when gluing together $n$ building blocks $B_1,...,B_n$ along Weierstrass edges, we can define a hyperelliptic involution $\gamma_{\sqcup_i B_i}$ on the total surface  by defining $\gamma_{\sqcup_i B_i}$ piecewise to be whichever $\gamma_{B_i}$ is defined.  The involution $\gamma_{\sqcup_i B_i}$ fixes $2(g_{B_1} + ... + g_{B_n})+2$ points.  Since every invariant component diagram $\mathcal{D}^{per}_{(k,p,m)}$ or $\mathcal{D}^{min}_{(k,p,m)}$ is a tree, the genus of any surface associated to such a diagram is the sum of the genera of the building blocks.  Consequently $\gamma_{\sqcup_i B_i}$ fixes $2g+2$ points, where $g$ is the genus of the total surface.  

The integer $k$ is the total number of half-edges (counting a full edge as two half-edges) in the invariant component diagram.  
By Lemma  \ref{l:totalConeAngle}, the total cone angle (the sum over the cone points of $S$ of the cone angle at that point) of $S$ is $2\pi k$.  A translation surface in the stratum $\mathcal{H}(2g-2)$ has total cone angle $2\pi(2g-1)$ and a translation surface in the stratum $\mathcal{H}(g-1,g-1)$ has total cone angle $2\pi(2g)$.  Then by Lemma \ref{l:numberOfConePts}, if $k$ is odd, the surface $S$ has genus $\frac{k+1}{2}$ and is in $\mathcal{H}^{hyp}(k-1)$; if $k$ is even, the surface $S$ has genus $\frac{k}{2}$ and is in $\mathcal{H}^{hyp}(\frac{k}{2}-1,\frac{k}{2}-1)$.   
  \end{proof}

\begin{proof}[Proof of Theorem \ref{t:maintheorem}]
The upper bounds in Theorem \ref{t:maintheorem} are given in Proposition \ref{p:upperBounds}.   Corollary \ref{c:realizable} proves that given any hyperelliptic connected component ($\mathcal{H}^{hyp}(2g-2)$ or $\mathcal{H}^{hyp}(g-1,g-1)$) for the moduli space of translation surfaces of genus $g \geq 2$, and any pair $(p,m)$ satisfying the upper bounds, there exists a translation surface in that hyperelliptic connected component with precisely $p$ periodic components and $m$ minimal components. 

Although the case $g=1$, $m=0$, $p=2$ satisfies the bound $3m+2p-1 \leq 2g$, no translation surface realizes this pair, since every flat torus consists of a single component.  In this case, the bound ``sees" the torus as consisting of two copies of $P_1$.  However, since the augmented ``cone points" on the boundaries of those cylinders are not true cone points (i.e. do not have cone angle greater than $2\pi$), the surface has no saddle connections, and thus the ``two cylinders" form a single periodic component.   
\end{proof}


\bibliographystyle{amsalpha}
\bibliography{InvariantComponentsBibliography}

\providecommand{\bysame}{\leavevmode\hbox to3em{\hrulefill}\thinspace}
\providecommand{\MR}{\relax\ifhmode\unskip\space\fi MR }
\providecommand{\MRhref}[2]{%
  \href{http://www.ams.org/mathscinet-getitem?mr=#1}{#2}
}
\providecommand{\href}[2]{#2}
\begin{thebibliography}{Mar70}

\bibitem[KZ03]{KontsevichZorich}
Maxim Kontsevich and Anton Zorich, \emph{Connected components of the moduli
  spaces of {A}belian differentials with prescribed singularities}, Invent.
  Math. \textbf{153} (2003), no.~3, 631--678. \MR{2000471 (2005b:32030)}

\bibitem[Mar70]{Markley}
Nelson~G. Markley, \emph{On the number of recurrent orbit closures}, Proc.
  Amer. Math. Soc. \textbf{25} (1970), 413--416. \MR{0256375 (41 \#1031)}

\bibitem[Mir95]{Miranda}
Rick Miranda, \emph{Algebraic curves and {R}iemann surfaces}, Graduate Studies
  in Mathematics, vol.~5, American Mathematical Society, Providence, RI, 1995.
  \MR{1326604 (96f:14029)}

\bibitem[MT02]{MasurTabachnikov}
Howard Masur and Serge Tabachnikov, \emph{Rational billiards and flat
  structures}, Handbook of dynamical systems, {V}ol.\ 1{A}, North-Holland,
  2002, pp.~1015--1089.

\bibitem[Nav08]{Naveh}
Yoav Naveh, \emph{Tight upper bounds on the number of invariant components on
  translation surfaces}, Israel J. Math. \textbf{165} (2008), 211--231.
  \MR{2403621 (2009a:37077)}

\bibitem[Zor06]{ZorichSurvey}
Anton Zorich, \emph{Flat surfaces}, Frontiers in number theory, physics, and
  geometry. {I}, Springer, Berlin, 2006, pp.~437--583.

\end{thebibliography}


\end{document}